\documentclass[a4paper,12pt]{article}

\usepackage{amsmath}
\usepackage{amsthm}
\usepackage{amssymb}
\usepackage[dvips]{graphicx}

\usepackage{epsfig}
\usepackage{psfrag}
\usepackage{subfigure}
\usepackage{graphicx}

\usepackage[mathscr,mathcal]{eucal}
\usepackage{enumerate}

\usepackage{t1enc}
\usepackage[latin2]{inputenc}

\def \C {\mathbb C}

\def \N {\mathbb N}

\def\cC{\mathcal{C}}

\def\cF{\mathcal{F}}

\def\cH{\mathcal{H}}

\def\cL{\mathcal{L}}

\newcommand{\expect}[1]{\ensuremath{\mathbf{E}\left(#1\right)}}

\def\one{1\!\!1}

\DeclareMathOperator*{\Dom}{Dom}
\DeclareMathOperator*{\Ran}{Ran}
\DeclareMathOperator*{\Ker}{Ker}

\renewcommand{\d}{\mathrm d}

\newcommand{\abs}[1]{\left|\,{#1}\,\right|}
\newcommand{\norm}[1]{\left\|\,{#1}\,\right\|}

\def \wt {\widetilde}

\def\wh{\widehat}

\newtheorem {theorem}{Theorem}

\newtheorem {lemma}{Lemma}

\newtheorem {proposition}{Proposition}

\newtheorem* {theorem*}{Theorem}
\newtheorem* {thm*}{Theorem}
\newtheorem* {lemma*}{Lemma}
\newtheorem* {lem*}{Lemma}
\newtheorem* {corollary*}{Corollary}
\newtheorem* {cor*}{Corollary}
\newtheorem* {proposition*}{Proposition}
\newtheorem* {prop*}{Proposition}
\newtheorem* {definition*}{Definition}
\newtheorem* {def*}{Definition}
\newtheorem* {conjecture*}{Conjecture}

\newtheorem* {theoremkv*} {Theorem KV}
\newtheorem* {theoremssc*} {Theorem SSC}
\newtheorem* {theoremgsc*} {Theorem GSC}

\theoremstyle{definition}
\newtheorem* {remark*}{Remark}
\newtheorem* {remarks}{Remarks}
\newtheorem* {rem*}{Remark}

\newtheorem* {assumption*}{Assumption}

\def\be{\begin{align}}
\def\ee{\end{align}}

\def\bea{\begin{eqnarray}}
\def\eea{\end{eqnarray}}

\newcommand{\tostoptop}{\stackrel{\mathrm{st.op.top.}}{\longrightarrow}}

\title{Relaxed sector condition}

\author{Ill\'es Horv\'ath\thanks{Institute of Mathematics, Budapest University of Technology, Egry J\'ozsef u.\ 1, 1111 Budapest, Hungary. E-mail: {\tt pollux@math.bme.hu}.} \and
B\'alint T\'oth\thanks{Institute of Mathematics, Budapest University of Technology, Egry J\'ozsef u.\ 1, 1111 Budapest, Hungary. E-mail: {\tt balint@math.bme.hu}.} \and
B\'alint Vet\H o\thanks{Institute for Applied Mathematics, Bonn University, Endenicher Allee 60, 53115 Bonn, Germany. E-mail: {\tt vetob@uni-bonn.de}.}}

\begin{document}

\maketitle

\begin{abstract}

In this note we present a new sufficient condition which guarantees martingale approximation and central limit theorem \emph{\`a la Kipnis\,--\,Varadhan} to hold for additive functionals of Markov processes. This condition, which we call the \emph{relaxed sector condition (RSC)} generalizes the strong sector condition (SSC) and the graded sector condition (GSC) in the case when the self-adjoint part of the infinitesimal generator acts diagonally in the grading. The main advantage being that the proof of the GSC in this case is more transparent and less computational than in the original versions. We also hope that the RSC may have direct applications where the earlier sector conditions don't apply. So far we don't have convincing examples in this direction.

\end{abstract}

\section{Introduction}
\label{s:intro}

The theory of central limit theorems for additive functionals of ergodic Markov processes via martingale approximation was initiated in the mid-1980-s with applications to tagged particle diffusion in stochastic interacting particle systems and various models of random walks in random environment.

The Markov process is usually assumed to be in a stationary and ergodic regime. We shall stick to these assumptions in the present note, too. There are however also other type of related results, see e.g.\ \cite{maxwell_woodroofe_00}, \cite{cuny_peligrad_12}, which use partly different techniques.

In their celebrated 1986 paper \cite{kipnis_varadhan_86}, C.\ Kipnis and S.\ R.\ S.\ Varadhan proved a central limit theorem for the reversible case with no assumptions other than the strictly necessary ones. For an early non-reversible extension see \cite{toth_86} where the martingale approximation was applied to a particular model of random walk in random environment.

The theory has since been widely extended by Varadhan and collaborators to include processes with a varying degree of non-reversibility. For a detailed account of these so-called \emph{sector conditions} and the different models they are applied to, see the surveys \cite{olla_01}, \cite{komorowski_landim_olla_12} and the more recent result \cite{horvath_toth_veto_12}.

In the present note, we introduce a new sector condition which we call the \emph{relaxed sector condition (RSC)}. Apart from appearing to be  interesting in its own right, it also provides a new, slightly improved  version of the \emph{graded sector condition (GSC)}, in the case when the self-adjoint part of the infinitesimal generator doesn't mix the subspaces of the graded Hilbert space. The proof presented here is less technical and more transparent.

\section{Setup, abstract considerations}
\label{s:setup}

We recall the non-reversible version of the abstract
Kipnis\,--\,Varadhan CLT for additive functionals of ergodic Markov processes, see \cite{kipnis_varadhan_86} and \cite{toth_86}.

Let $(\Omega, \cF, \pi)$ be a probability space: the state space of a \emph{stationary and ergodic} Markov process $t\mapsto\eta(t)$. We put ourselves in the Hilbert space $\cH:=\cL^2(\Omega, \pi)$. Denote the \emph{infinitesimal generator} of the semigroup of the process by $G$, which is
a well-defined (possibly unbounded) closed linear operator on $\cH$.

The adjoint $G^*$ is the infinitesimal generator of the semigroup of the reversed (also stationary and ergodic) process $\eta^*(t)=\eta(-t)$. It is assumed that $G$ and $G^*$ have a \emph{common core of definition} $\cC\subseteq\cH$. We denote the \emph{symmetric} and \emph{antisymmetric} parts of the generators $G$, $G^*$, by
\begin{align*}
S:=-\frac12(G+G^*),
\qquad
A:=\frac12(G-G^*).
\end{align*}
(We prefer to use the notation $S$ for the positive semidefinite operator defined above, so the infinitesimal generator will be written as $G=-S+A$.) These operators are also extended from $\cC$ by graph closure and it is assumed that they are well-defined self-adjoint, respectively, skew self-adjoint operators:
\begin{align*}
S^*=S\ge0, \qquad A^*=-A.
\end{align*}
Summarizing: it is assumed that the operators $G$, $G^*$, $S$ and $A$ have a common dense core of definition $\cC$.
Note that $-S$ is itself the infinitesimal generator of a Markovian semigroup on $\cL^2(\Omega,\pi)$, for which the probability measure $\pi$ is reversible (not just stationary). We assume that $-S$ is itself ergodic:
\begin{align*}
\Ker(S)=\{c1\!\!1 : c\in\C\}.
\end{align*}
We shall restrict ourselves to the subspace of codimension 1, orthogonal to the constant functions.

In the sequel the operators $(\lambda I+ S)^{\pm1/2}$, $\lambda\ge0$, will play an important r\^ole. These are defined by the spectral theorem applied to the self-adjoint and positive operator $S$. It is easy to see that $\cC$ is also a core for the operators $(\lambda I+ S)^{1/2}$, $\lambda\ge0$. The operators $(\lambda I+ S)^{-1/2}$, $\lambda>0$, are everywhere defined and bounded, with $\norm{(\lambda I+ S)^{-1/2}}\le\lambda^{-1/2}$. The operator $S^{-1/2}$ is defined on
\begin{align}
\label{conditionH-1}
\Dom(S^{-1/2}):=
\{f\in\cH:
\norm{S^{-1/2}f}^2
:=
\lim_{\lambda\to0}\norm{(\lambda I + S)^{-1/2}f}^2
<
\infty
\}
=
\Ran(S^{1/2}).
\end{align}
We shall refer to \eqref{conditionH-1} as the \emph{$H_{-1}$-condition}.

Let $f\in\cH$, such that $(f, \one) = \int_\Omega f\,\d\pi=0$. We ask about CLT/invariance principle for
\begin{align}
\label{rescaledintegral}
N^{-1/2}\int_0^{Nt} f(\eta(s))\,\d s
\end{align}
as $N\to\infty$.

We denote by $R_\lambda$ the resolvent of the semigroup $s\mapsto e^{sG}$:
\begin{align}
\label{resolvent_def}
R_\lambda
:=
\int_0^\infty e^{-\lambda s} e^{sG}\d s
=
\big(\lambda I-G\big)^{-1}, \qquad \lambda>0,
\end{align}
and given $f\in\cH$ as above, we will use the notation
\begin{align*}
u_\lambda:=R_\lambda f.
\end{align*}

The following theorem is direct extension to general non-reversible setup of the Kipnis\,--\,Varadhan theorem from  \cite{kipnis_varadhan_86}. It yields the efficient martingale approximation of the additive functional \eqref{rescaledintegral}. To the best of our knowledge this non-reversible extension appears first in \cite{toth_86}.

\begin{theoremkv*}
\label{thm:kv}
With the notation and assumptions as before, if the following two limits hold in $\cH$:
\begin{align}
\label{conditionA}
&
\lim_{\lambda\to0}
\lambda^{1/2} u_\lambda=0,
\\
\label{conditionB}
&
\lim_{\lambda\to0} S^{1/2} u_\lambda=:v\in\cH,
\end{align}
then
\begin{align*}
\sigma^2
:=
2\lim_{\lambda\to0}(u_\lambda,f)
=
2\norm{v}^2
\in
[0,\infty),
\end{align*}
exists, and there also exists a zero mean, $\cL^2$-martingale $M(t)$ adapted to the filtration of the Markov process $\eta(t)$ with stationary and ergodic increments and variance
\begin{align*}
\expect{M(t)^2}=\sigma^2t
\end{align*}
such that
\begin{align*}
\lim_{N\to\infty} N^{-1} \expect{\big(\int_0^N
f(\eta(s))\,\d s-M(N)\big)^2} =0.
\end{align*}
In particular, if $\sigma>0$, then the finite dimensional marginal distributions of the rescaled process $t\mapsto \sigma^{-1} N^{-1/2}\int_0^{Nt}f(\eta(s))\,\d s$ converge to those of a standard $1d$ Brownian motion.
\end{theoremkv*}

\begin{remarks}

\begin{enumerate}[$\circ$]

\item
For the historical record it should be mentioned that the idea of martingale approximation and an early variant of this theorem under the much more restrictive condition
$f\in\mathrm{Ran}(G)$, appears in \cite{gordin_lifshits_78}. For more exhaustive historical account and bibliography of the problem see the recent monograph \cite{komorowski_landim_olla_12}.

\item
The reversible case, when $A=0$, was considered in the celebrated paper \cite{kipnis_varadhan_86}. In that case conditions \eqref{conditionA} and \eqref{conditionB} are equivalent. The proof of the Theorem KV in the reversible case relies on spectral calculus.

\item
Conditions \eqref{conditionA} and \eqref{conditionB} of Theorem KV are jointly equivalent to the following
\begin{align}
\label{conditionC}
\lim_{\lambda,\lambda'\to0}(\lambda+\lambda')(u_\lambda,u_{\lambda'})=0.
\end{align}
Indeed, straightforward computations yield:
\begin{align*}
(\lambda+\lambda')(u_\lambda,u_{\lambda'}) =
\norm{S^{1/2}(u_\lambda-u_{\lambda'})}^2 + \lambda \norm{u_\lambda}^2 +
\lambda' \norm{u_{\lambda'}}^2.
\end{align*}

\item
The non-reversible formulation appears -- in discrete-time Markov chain, rather than continuous-time Markov process setup and with condition \eqref{conditionC} -- in \cite{toth_86} where it was applied, with bare hands computations, to obtain CLT for a particular random walk in random environment. Its proof mainly follows the original proof of the Kipnis\,--\,Varadhan theorem from \cite{kipnis_varadhan_86} with the difference that spectral calculus is replaced by resolvent calculus.

\item
In continuous-time Markov process setup, it was formulated in \cite{varadhan_96} and applied to tagged particle motion in non-reversible zero mean exclusion processes. In this paper, the \emph{strong sector condition (SSC)} was formulated, which, together with the $H_{-1}$-condition \eqref{conditionH-1} on the function $f\in\cH$, provide sufficient conditions for \eqref{conditionA} and \eqref{conditionB} of Theorem KV to hold.

\item
In \cite{sethuraman_varadhan_yau_00}, the so-called \emph{graded sector condition (GSC)} was formulated and Theorem KV was applied to tagged particle diffusion
in general (non-zero mean) non-reversible exclusion processes, in $d\ge3$. The fundamental ideas related to the GSC have their origin partly in \cite{landim_yau_97}.

\item
For a list of applications of Theorem KV together with the SSC and GSC, see the surveys \cite{olla_01}, \cite{komorowski_landim_olla_12}, and for a more recent application of the GSC to the so-called \emph{myopic self-avoiding walks} and \emph{Brownian polymers}, see \cite{horvath_toth_veto_12}.
\end{enumerate}
\end{remarks}

\section{Sector conditions}
\label{s:sc}

In subsection \ref{ss:ssc&gsc} we recall the SSC and the GSC. In subsection \ref{ss:rsc} we formulate the RSC, which is the main abstract result of this note and, as a consequence, a slightly improved version of GSC. In further sections we first prove the RSC, then we show how the SSC and GSC follow in a very natural way from RSC. The main gain is not in slightly weakening the conditions but in simplifying the proof of GSC.

\subsection{Strong and graded sector conditions}
\label{ss:ssc&gsc}

From abstract functional analytic considerations, it follows that the $H_{-1}$-condition \eqref{conditionH-1} together with the following bound jointly imply \eqref{conditionC}, and hence the martingale approximation and CLT of Theorem KV:
\begin{align}
\label{conditionD}
\sup_{\lambda>0}
\norm{ S^{-1/2} G u_\lambda} < \infty.
\end{align}

\begin{theoremssc*}
\label{thm:ssc}
With notations as before, if there exists a constant $C<\infty$ such that for any $\varphi, \psi\in \cC$, the common core of $S$ and $A$,
\begin{equation}
\label{ssc}
\abs{(\psi,A\varphi)}^2 \le C^2 (\psi, S\psi) (\varphi, S \varphi),
\end{equation}
then for any $f\in\cH$ for which \eqref{conditionH-1} holds, \eqref{conditionD} also follows. So for every function $f$ for which \eqref{conditionH-1} holds, the martingale approximation and CLT of Theorem KV applies automatically.
\end{theoremssc*}

\begin{remark*}

\begin{enumerate}[$\circ$]

\item
Condition \eqref{ssc} is equivalent to requiring that the operator $S^{-1/2} A S^{-1/2}$ defined on the dense subspace $S^{1/2}\cC:=\{S^{1/2}\varphi: \varphi\in\cC\}$ be bounded in norm by the constant $C$. Hence, by continuous extension, condition \eqref{ssc} is the same as
\begin{equation}
\label{ssc2}
\norm{S^{-1/2}AS^{-1/2}}\le C<\infty.
\end{equation}

\end{enumerate}

\end{remark*}
For the GSC, assume that the Hilbert space $\cH=\cL^2(\Omega, \pi)$ is graded
\begin{align}
\label{grading}
\cH=\overline{\oplus_{n=0}^\infty\cH_n}
\end{align}
where $\cH_0$ is the 1-dimensional subspace of constant functions. Since we work with functions $f$ for which $\int_\Omega f\,\d\pi=0$, we exclude the subspace $\cH_0$ from $\cH$ without changing the notation (and thus abusing it slightly).

Also, assume that the infinitesimal generator is consistent with the grading in the following sense:
\begin{align}
\label{Sgrading}
&
S=\sum_{n\ge1} S_{n,n},
&&
S_{n,n}:\cH_n\to\cH_n,
&&
S_{n,n}^*=S_{n,n}\ge0,
\\
\label{Agrading}
&
A=\sum_{ \substack{m,n\ge1 \\ \abs{n-m}\le r} } A_{m,n},
&&
A_{m,n}:\cH_n\to\cH_m,
&&
A_{m,n}^*=-A_{n,m}
\end{align}
where $r$ is a fixed finite integer. This means that the operator $S$ acts diagonally on the grading (doesn't mix the subspaces $\cH_n$), while the operator $A$ only mixes subspaces whose indices are closer than a fixed finite amount.  The operators $S_{n,n}$ and $A_{m,n}$ are not necessarily bounded. $\cC_n= \cC\cap \cH_n$ is a common core for them.

\begin{theoremgsc*}
\label{thm:gsc}
Let the Hilbert space and the infinitesimal generator be graded in the sense specified above. If there exist $\kappa<\infty$, $\beta<1$ and $C<\infty$ such that for any $n,m\in\N$ and $\psi_m\in \cC_m$, $\varphi_n\in\cC_n$ the following bounds hold:
\begin{align}
\label{gsc}
\abs{(\psi_m, A_{m,n}\varphi_n)}^2
\le
C^2
\left(\delta_{m,n}n^{2\kappa}+ (1-\delta_{m,n})n^{2\beta}\right)
(\psi_m, S_{m,m} \psi_m) (\varphi_n, S_{n,n} \varphi_n),
\end{align}
then, for any function $f\in \oplus_{n=0}^\infty \cH_n$ (no closure!), for which \eqref{conditionH-1} holds, \eqref{conditionD} also follows. As a consequence, for these functions the martingale approximation and CLT of Theorem KV hold.

The statement remains valid for $\beta=1$ if $C$ is sufficiently small.
\end{theoremgsc*}

\begin{remarks}

\begin{enumerate}[$\circ$]

\item
Condition \eqref{gsc} is equivalent to requiring that the operators
\[
S_{m,m}^{-1/2} A_{m,n} S_{n,n}^{-1/2}: \cH_n \to \cH_m
\]
be bounded, with norm bounds
\begin{equation}
\norm{S_{m,m}^{-1/2} A_{m,n} S_{n,n}^{-1/2}}
\le
C
\left(\delta_{m,n}n^{\kappa}+ (1-\delta_{m,n})n^{\beta}\right)
\end{equation}

\item
There exists a stronger version of Theorem GSC, where it is not required that the self-adjoint part acts diagonally on the grading, see \cite{sethuraman_varadhan_yau_00}, \cite{komorowski_landim_olla_12}, or \cite{horvath_toth_veto_12} for the sharpest formulation. Our simplified proof seems to work smoothly only in the case when $S$ acts diagonally in the grading.

\end{enumerate}

\end{remarks}

\subsection{Relaxed sector condition}
\label{ss:rsc}

Let, as before,  $\cC\subset\cH$ be a common core for the operators  $G$, $G^*$, $S$ and $A$. Note that for any $\lambda>0$, $\cC\subseteq\Dom((\lambda I + S)^{1/2})$ and the subspace
\[
(\lambda I+S)^{1/2}\cC:=
\{(\lambda I+S)^{1/2}\varphi:\varphi\in\cC\}
\]
is dense in $\cH$. The operators
\begin{equation}
\label{Blambda_def}
B_\lambda: (\lambda I +S)^{1/2}\cC \to\cH,
\quad
B_\lambda:=(\lambda I + S)^{-1/2} A (\lambda I + S)^{-1/2},
\qquad
\lambda>0
\end{equation}
are densely defined and skew-Hermitian, and thus closable. Actually it is the case that they are not only skew-Hermitian, but essentially skew self-adjoint on $(\lambda I+S)^{1/2}\cC$. Indeed, let $\chi\in\cC$, $\varphi=(\lambda I+S)^{1/2}\chi$ and $\psi\in\cH$, then
\[
(\psi, (I\pm B_\lambda) \varphi)
=
((\lambda I+ S)^{-1/2}\psi, (\lambda I + S \pm A) \chi).
\]
So,  $\psi\perp\Ran(I\pm B_\lambda)$ implies $(\lambda I+ S)^{-1/2}\psi \perp \Ran (\lambda I + S \pm A)$ and thus, since the operators $S\pm A$ are Hille-Yosida-type, $(\lambda I+ S)^{-1/2}\psi= 0$, and consequently $\psi=0$ holds. That is $\Ran(I\pm B_\lambda)$ is dense in $\cH$. By slight abuse of notation we shall denote by the same symbol $B_\lambda$ the skew self-adjoint operators obtained by closure of the operators defined in \eqref{Blambda_def}.

The main point of the following theorem is that if there exists another skew self-adjoint operator $B$, \emph{formally} identified as \begin{equation}
\label{B_def}
B:=S^{-1/2} A S^{-1/2},
\end{equation}
and a sufficiently large subspace on which the sequence of operators $B_\lambda$ converges pointwise (strongly) to $B$, as $\lambda\to0$, then, the  $H_{-1}$-condition \eqref{conditionH-1} implies    \eqref{conditionA} and \eqref{conditionB}, and thus the martingale approximation and CLT of Theorem KV follow.

\begin{theorem}[\bf Relaxed sector condition]
\label{thm:rsc}
Assume that there exist a subspace $\wt\cC\subseteq \cap_{\lambda>0} \Dom (B_\lambda)$ which is still dense in $\cH$ and an operator $B:\wt\cC\to \cH$ which is essentially skew self-adjoint and such that for any vector $\varphi\in \wt\cC$
\begin{align}
\label{Blambdalimit}
\lim_{\lambda\to 0}\norm{B_\lambda\varphi-B\varphi}=0.
\end{align}
Then, the  $H_{-1}$-condition \eqref{conditionH-1} implies \eqref{conditionA} and \eqref{conditionB}, and thus the martingale approximation and CLT of Theorem KV follow.
\end{theorem}

\begin{remarks}
\begin{enumerate}[$\circ$]

\item
Finding the appropriate subspace $\wt\cC$ and defining the skew-Hermitian operator $B:\wt\cC\to\cH$ comes naturally. The difficulty in applying this criterion lies in proving that the operator $B$ is not just skew-Hermitian, but actually skew self-adjoint. That is, proving that \begin{equation}
\label{dense_range}
\overline{\Ran(I\pm B)}=\cH.
\end{equation}
This is the counterpart of \emph{the basic criterion of self-adjointness}. See e.g.\ Theorem VIII.3. of \cite{reed_simon_vol1_vol2_75}. Checking this is typically not easy in concrete cases.

\item
The statement and the proof of this theorem show close similarities with the Trotter-Kurtz theorem. See Theorem 2.12 in \cite{liggett_85}.

\item
Theorem SSC follows directly: In this case the operator $B$ is actually \emph{bounded} and thus automatically skew self-adjoint, not just skew-Hermitian. In order to see \eqref{Blambdalimit} note that
\begin{align}
\label{strongconvergence}
B_\lambda
=
S^{1/2}(\lambda I + S)^{-1/2} B  S^{1/2}(\lambda I + S)^{-1/2}
\tostoptop
B,
\end{align}
where $\tostoptop$ denotes convergence in the strong operator topology.
\end{enumerate}
\end{remarks}
As a direct consequence we formulate a slightly stronger version of Theorem GSC. The main advantage is actually in the proof: our proof is considerably less computational, more transparent and natural than the original one from \cite{sethuraman_varadhan_yau_00}, reproduced in a streamlined way in \cite{olla_01} and \cite{komorowski_landim_olla_12}.

Assume the setup of Theorem GSC:
the grading of the Hilbert space and the infinitesimal generator $G$ acting consistently with the grading: \eqref{grading}, \eqref{Sgrading}, \eqref{Agrading}.

\begin{proposition}[\bf GSC from RSC]
\label{prop:gsc_from_rsc}
If there exist two positive non-decreasing sequences $d_n$ and $c_n$ such that
\begin{equation}
\label{sequences}
d_n<\infty,
\qquad
\sum_{n=1}^\infty c_n^{-1} =\infty,
\end{equation}
and for any $n,m\in\N$ and $\psi_m\in \cC_m$, $\varphi_n\in\cC_n$ the following bounds hold:
\begin{align}
\label{gsc2}
\abs{(\psi_m, A_{m,n}\varphi_n)}^2
\le
\left(\delta_{m,n}d_n+ (1-\delta_{m,n})c_n\right)
(\psi_m, S_{m,m} \psi_m) (\varphi_n, S_{n,n} \varphi_n),
\end{align}
then the conditions of Theorem \ref{thm:rsc} hold with $\wt\cC=\oplus_{n=1}^\infty \cH_n$ (no closure!).
\end{proposition}

\section{Proofs}
\label{s:proofs}

\subsection{Proof of Theorem \ref{thm:rsc}}
\label{ss:proof_of_the_theorem}

Since the operators $B_\lambda$, $\lambda>0$, defined in \eqref{Blambda_def} are a priori and the operator $B$ is by assumption skew self-adjoint, we can define the following bounded operators (actually contractions):
\begin{align*}
&
K_\lambda:=(I-B_\lambda)^{-1},
&&
\norm{K_\lambda}\le1,
&&
\lambda>0,
\\
&
K:=(I-B)^{-1},
&&
\norm{K}\le1.
\end{align*}
Hence,  we can write the resolvent \eqref{resolvent_def} as
\begin{align}
\label{resolvent_master}
R_\lambda
=
(\lambda+S)^{-1/2}
K_\lambda
(\lambda+S)^{-1/2}.
\end{align}

\begin{lemma}
\label{lem:Klambdastcvg}
Assume that the sequence of bounded operators $K_\lambda$ converges in the strong operator topology:
\begin{align}
\label{Klambdastcvg}
K_\lambda\tostoptop K,
\qquad \text{as}\qquad\lambda\to0.
\end{align}
Then for any $f$ satisfying the $H_{-1}$-condition \eqref{conditionH-1},  \eqref{conditionA} and \eqref{conditionB} hold.
 \end{lemma}

\begin{proof}
From the spectral theorem applied to the self-adjoint operator $S$, it is obvious that
\begin{align}
\label{llambdaS}
&
\norm{\lambda^{1/2}(\lambda+S)^{-1/2}}\le1,
&&
\lambda^{1/2}(\lambda+S)^{-1/2}\tostoptop 0,
\\
\label{SlambdaS}
&
\norm{S^{1/2}(\lambda+S)^{-1/2}}\le1,
&&
S^{1/2}(\lambda+S)^{-1/2}\tostoptop I,
\end{align}

By condition \eqref{conditionH-1} we can write
\begin{align*}
f = S^{1/2} g
\end{align*}
with some $g\in\cH$.  Now,  using \eqref{resolvent_master}, we get
\begin{align}
\label{lambdaulambda}
\lambda^{1/2}u_\lambda
&=
\lambda^{1/2}(\lambda+S)^{-1/2}
K_\lambda
(\lambda+S)^{-1/2}S^{1/2} g,
\\
\label{Sulambda}
S^{1/2}u_\lambda
&=
S^{1/2}(\lambda+S)^{-1/2}
K_\lambda
(\lambda+S)^{-1/2}S^{1/2} g.
\end{align}
From \eqref{Klambdastcvg}, \eqref{lambdaulambda}, \eqref{Sulambda}, \eqref{llambdaS} and \eqref{SlambdaS}, we readily get \eqref{conditionA} and \eqref{conditionB} with
\begin{align*}
v=Kg.
\end{align*}
\end{proof}

In the next lemma, we formulate a sufficient condition
for \eqref{Klambdastcvg} to hold. This is reminiscent of Theorem VIII.25(a) from \cite{reed_simon_vol1_vol2_75}:

\begin{lemma}
\label{lem:strrescvg}
Let $B_n$, $n\in\N$, and $B=B_\infty$ be densely defined closed operators over the Hilbert space $\cH$. Assume that
\begin{enumerate}[(i)]
\item
Some (fixed) $\mu\in\C$ is in the resolvent set of all operators $B_n$,
$n\le\infty$, and
\begin{align}
\label{unifbound}
\sup_{n\le\infty}\norm{(\mu I - B_n)^{-1}}<\infty.
\end{align}

\item
There is a dense subspace $\wt\cC\subseteq\cH$ which is a core for $B_\infty$ and $\wt\cC\subseteq\Dom(B_n)$, $n<\infty$, such that for all $\wt h\in\wt\cC$:
\begin{align}
\label{stcvgoncore}
\lim_{n\to0}\norm{B_n \wt h-B \wt h}=0.
\end{align}
\end{enumerate}
Then
\begin{align}
\label{strescvg}
(\mu I - B_n)^{-1}
\tostoptop
(\mu I - B)^{-1}.
\end{align}
\end{lemma}

\begin{proof}
Since $\wt\cC$ is a core for the densely defined closed operator $B$ and $\mu$ is in the resolvent set of $B$, the subspace
\begin{align*}
\wh{\cC}:=\{\wh h=(\mu I  - B) \wt h\,:\, \wt h\in\wt\cC\}
\end{align*}
is dense in $\cH$. Thus, for any $\wh h$ from this dense subspace, we have
\begin{align*}
\big\{(\mu I - B_n)^{-1}-(\mu I - B)^{-1}\big\} \wh h =
(\mu I-B_n)^{-1}(B_n \wt h-B \wt h) \to0,
\end{align*}
due to \eqref{unifbound} and \eqref{stcvgoncore}. Using again \eqref{unifbound}, we conclude \eqref{strescvg}.
\end{proof}

Putting Lemmas \ref{lem:Klambdastcvg} and \ref{lem:strrescvg} together, we obtain Theorem \ref{thm:rsc}.
\hfill
\qed

\subsection{Proof of Proposition \ref{prop:gsc_from_rsc}}

Let
\begin{align*}
\wt\cC=\oplus_{n=1}^\infty\cH_n.
\end{align*}
Note that there is no closure of the orthogonal sum on the right hand side. Then the operator $B=S^{-1/2}AS^{-1/2}$ is defined on $\wt\cC$ and is graded as
\begin{align*}
&
B=\sum_{ \substack{m,n\ge1 \\ \abs{n-m}\le r} } B_{m,n},
&&
B_{m,n}:\cH_n\to\cH_m,
&&
B_{m,n}:=S_{m,m}^{-1/2}A_{m,n}S_{m,m}^{-1/2},
&&
B_{m,n}^*=-B_{n,m}.
\end{align*}
Indeed, due to \eqref{gsc2}
\begin{align}
\label{B-norms}
\norm{B_{n,m}}\le \delta_{m,n}d_n+ (1-\delta_{m,n})c_n.
\end{align}
The operator $B:\wt\cC\to\wt\cC$ is clearly skew-Hermitian. In order to prove that it is actually essentially skew self-adjoint we have to check \eqref{dense_range}.

For $\varphi\in\cH$ we use the notation
\begin{align*}
\varphi=(\varphi_1,\varphi_2, \dots),
\qquad
\varphi^n:=(\varphi_1,\varphi_2, \dots, \varphi_n,0,0,\dots).
\end{align*}

In order to simplify the notation in the forthcoming argument we assume that $r=1$. The cases with $r>1$ are done exactly the same way, only notation becomes heavier.

Assume $\varphi \perp \Ran (I-B)$, then
\begin{align*}
0
=
(\varphi, (I-B) \varphi^n)
=
\norm{\varphi^n}^2 -(\varphi_{n+1}, B_{n+1,n} \varphi_n).
\end{align*}
Hence, by \eqref{B-norms} and letting $n$ so large that $\norm{\varphi^n}^2\ge\norm{\varphi}^2/2$,
\begin{align*}
\norm{\varphi_n}^2 + \norm{\varphi_{n+1}}^2
\ge
\frac{2}{c_n} \norm{\varphi^n}^2
\ge
\frac{1}{c_n}  \norm{\varphi}^2.
\end{align*}
Summing over $n$ we obtain that $\varphi=0$. This implies that $\Ran(I-B)$ is dense in $\cH$. Identical argument works for $\Ran(I+B)$. This proves \eqref{dense_range} and $B$ is indeed essentially skew self-adjoint on $\wt\cC$.

Checking condition \eqref{stcvgoncore} is done exactly like in \eqref{strongconvergence}:
\begin{align*}
(B_\lambda)_{m,n}
&
=
S_{m,m}^{1/2}(\lambda I_{m,m}+ S_{m,m})^{-1/2} B_{m,n} S_{n,n}^{1/2}(\lambda I_{n,n}+ S_{n,n})^{-1/2}
\\[8pt]
&
\tostoptop B_{m,n},
\end{align*}
as $\lambda\to0$, since $\norm{B_{m,n}}<\infty$ and $S_{m,m}^{1/2}(\lambda I_{m,m}+ S_{m,m})^{-1/2} \tostoptop I_{m,m}$.
\hfill
\qed

\vskip2cm

\noindent
{\bf Acknowledgement:}
This work was supported by OTKA (Hungarian
National Research Fund) grant K 100473 and
by the grant T\'AMOP -- 4.2.2.B--10/1--2010-0009.

\end{document}